\pdfoutput=1
\documentclass[12pt]{article}
\usepackage[margin=1in]{geometry}
\usepackage{amsmath}
\usepackage{amsfonts}
\usepackage{amssymb}
\usepackage{amsthm}
\usepackage{esint}
\usepackage{mathtools}
\usepackage{array,makecell,longtable}
\usepackage{titling}
\usepackage{authblk}
\usepackage{commath}
\usepackage{xcolor}
\usepackage{graphicx}
\usepackage[toc,title,page]{appendix}
\usepackage{subcaption}
\usepackage{bbm}
\usepackage{hyperref,cite}
\usepackage{color,soul}
\usepackage{bm}
\usepackage{mathrsfs}

\usepackage{tikz}
\usepackage{enumitem}
\newtheorem{remark}{Remark}[section]
\newtheorem{theorem}[remark]{Theorem}

\newtheorem{lemma}[remark]{Lemma}

\newtheorem{definition}[remark]{Definition}

\numberwithin{equation}{section}
\allowdisplaybreaks

\newcommand{\p}{\partial} 

\title{Simultaneously decoding the unknown stationary state and function parameters for mean field games}
\author[1,*]{Hongyu Liu}
\author[1,$\dagger$]{Catharine W. K. Lo}
\affil[1]{Department of Mathematics, City University of Hong Kong}
\affil[*]{hongyu.liuip@gmail.com, hongyliu@cityu.edu.hk}
\affil[$\dagger$]{wingkclo@cityu.edu.hk}
\date{}

\begin{document}
\maketitle

\begin{abstract}
    Mean field games (MFGs) offer a versatile framework for modeling large-scale interactive systems across multiple domains. This paper builds upon a previous work, by developing a state-of-the-art unified approach to decode or design the unknown stationary state of MFGs, in addition to the underlying parameter functions governing their behavior. This result is novel, even in the general realm of inverse problems for nonlinear PDEs. By enabling agents to distill crucial insights from observed data and unveil intricate hidden structures and unknown states within MFG systems, our approach surmounts a significant obstacle, enhancing the applicability of MFGs in real-world scenarios. This advancement not only enriches our understanding of MFG dynamics but also broadens the scope for their practical deployment in various contexts.

    \medskip
		
		

\end{abstract}

\section{Introduction}

\subsection{Mathematical Setup}

In recent years, mean field games (MFGs) have emerged as a powerful framework for analyzing the behaviors of large populations of interacting agents. This interdisciplinary approach transcends its origins in statistical mechanics, physics and quantum chemistry, providing a comprehensive perspective for studying complex systems where individual decisions collectively shape outcomes. MFGs are employed to model and investigate a diverse range of intricate physical systems. For instance, these models can simulate crowd movements, where a pedestrian's motion is influenced by the surrounding crowd and traffic flows \cite{Achdou2021econs,MFGCrowd,MFGCrowd+Econs,MFGCar2,MFGAutoCar2}. Similar models are also used in economic and financial markets \cite{MFGCrowd+Econs,CarmonaDelarue2018_1,Lacker2019finance,sequeira2024liquiditypoolsmeanfield,guan2024manyinsurerrobustgamesreinsurance}, where macroeconomic quantities are influenced by microeconomic behaviors of interacting agents. MFGs can also be used to study how social influence and collective behaviors or trends propagate through social networks, affecting individuals' opinions, decisions and behaviors \cite{MFGSocialNetwork}. Due to their capacity to capture emergent phenomena arising from interactions and strategic decision-making made by a large cohort of agents, MFGs have gained significant attention and have become a vibrant field of research.

The concept of mean field games traces its roots back to the groundbreaking work of Lasry and Lions \cite{LasryLions1, LasryLions2}, as well as that of Caines, Huang, and Malham\'e \cite{huang2006large}. Since their inception, mean field games have demonstrated remarkable efficacy in elucidating equilibrium states, optimal strategies, and the emergence of global patterns across diverse domains. Consequently, mean field games have attracted significant interest and attention, evolving into a vibrant and burgeoning field of research.

In such games, each player maximizes their expectations by selecting a strategy that relies not only on their own attributes but also on the distribution density of other players. Every agent is presumed to possess infinite rationality, attributing the same qualities to their peers, while also assuming that their individual strategy minimally impacts the equilibrium.
No player has an incentive to to deviate from the planned control unilaterally. However, in large population dynamic games, it becomes impractical for a player to gather detailed information about the state and strategies of all other players.

Hence, the core of MFGs is rooted in the concept where agents engage in a mean field Nash equilibrium. In this scenario, instead of foreseeing the reactions of others to their actions, individuals, in the extensive limit of a large population, only need to execute strategies based on the distribution of other players. By examining the aggregate behaviors of an infinitely expansive population, MFGs offer a macroscopic viewpoint that captures collective dynamics and strategic interactions of agents, abstracting from individual peculiarities. This macroscopic approach of MFGs, focusing on averaged characteristics, enables a simplified yet profound analysis of intricate systems without the need for approximations.

The mean field equilibrium is typically represented by a nonlinearly coupled pair of parabolic equations, where the first equation, a Hamilton-Jacobi-Bellman (HJB) equation, evolves backward in time, while the second equation, a Kolmogorov-Fokker-Planck (KFP) equation, progresses forward in time. This system of mean field games is structured as follows:
   \begin{equation}\label{eq:MFG0}
    \begin{cases}
        -\partial_t u -\Delta u + \mathcal{H}(x,t,\nabla u,m) = F(x,u,m) &\quad \text{in }Q:=\overline{\Omega}\times[0,T],\\
        \partial_t m -\Delta m - \nabla\cdot(m\nabla_p \mathcal{H}(x,t,\nabla u,m)) = 0  &\quad \text{in }Q\\
        u(x,T)=u_T,\quad m(x,0)=f(x)&\quad \text{in }\Omega.
    \end{cases}
\end{equation} 
Here, the unknowns $u$ and $m$ are scalar functions depending on time $t\in[0, T ]$ and state space $x\in\Omega$, where $\Omega$ is a bounded Lipschitz domain in the Euclidean space $\mathbb{R}^n$, with $n\in\mathbb{N}$. 
The state variable $x$ embodies various characteristics relevant to the problem at hand, such as physical location, social status, quantities of specific resources, among others.

Given the substantial population size, in the sense that the number of players approaches infinity, it is reasonable to consider just a mean-representative player. The MFG system \eqref{eq:MFG0} then captures the dynamics of this agent's behavior and the evolution of the population density in the following manner:
the convexity of the nonlinear Hamiltonian $\mathcal{H}$ with respect to the third variable means that the first HJB equation is associated with an optimal control problem and is interpreted as the value function $u$ associated with the mean-representative player. This representative player then interacts with the other players as a mass in the sense that the player does not consider the other players individually, but rather, simply the empirical density, denoted as $m$, of the positions of the players in the state space as a mass. Furthermore, as the number of players becomes very large, it is reasonable to assume that $m$ becomes deterministic and is unaffected by the behavior of any single player. Therefore, $m(\cdot,t)\in\mathcal{P}(\Omega)$, where $\mathcal{P}(\Omega)$ denotes the set of Borel probability measures on $\Omega$. 
In this system, $F$ represents the running cost function signifying agent-population interactions; 
$f$ denotes the initial population distribution, with all functions being real-valued.

The characterization of system \eqref{eq:MFG0} as a Nash equilibrium can be understood as follows: 
given that an individual player is ``infinitesimal" in comparison to the collective of other agents, his/her deviation from optimal behavior does not alter population dynamics. Consequently, the behavior of other agents and their collective time-dependent distribution $m$ can be considered fixed after optimization. This concept aligns with the notion of Nash equilibrium, where all players adopt optimal strategies while assuming others' choices remain constant.

Henceforth, it is vital for the individual player to acquire an understanding of the optimal state of the population, while having full comprehension of his/her individual costs $F$, in order to decode the internal system and formulate their optimal strategy. Conversely, one can also inversely design and control the dynamics of the system, based on an expected output. In a previous work \cite{LiuLoZhang2024decodingMFG}, we have successfully determined various intrinsic components of the system \eqref{eq:MFG0}, including the running cost $F$ and terminal cost $u_T$.

In this work, we tackle an unresolved issue in \cite{LiuLoZhang2024decodingMFG}, following the scenario described in that work. Let us briefly recall the setup of the inverse problem in \cite{LiuLoZhang2024decodingMFG}. Assume that the system \eqref{eq:MFG0} admits a stable stationary solution $U:=(u_0, m_0)$, the exact definition of which will be given later in Definition \ref{stable}. We only have access to information about $(u_0, m_0)$ on the boundary $\Sigma=\partial\Omega$. However, the values of $u_0$ and $m_0$ within the interior of $\Omega$ remain unknown, since the interior of $\Omega$ is inaccessible. Then for the inverse problem, we perturb these boundary values as \[u_0|_{\Sigma\times[0,T]}+\epsilon \tilde{u}, \quad m_0|_{\Sigma\times[0,T]}+\epsilon \tilde{m},\] where $\epsilon$ is a sufficiently small positive constant, and $\tilde{u},\tilde{m}\in C^{2+\alpha,1+\frac{\alpha}{2}}(Q)$. Corresponding to these perturbed boundary data, the MFG system yields a solution $(u_{0,\epsilon},m_{0,\epsilon})$ in $\tilde{\Omega}\times[0,T]$, which remains unknown in the interior of $\Omega$, but known on $\Sigma$. This known data essentially constitutes the measured Cauchy dataset:
\begin{equation}\mathcal{M}_\tau=\left.(u_{0,\epsilon}(x,t),m_{0,\epsilon}(x,t)\text{ and their related functions for $\epsilon \in (0,\tau)$})\right|_{\Sigma\times[0,T]},\end{equation} for $\tau$ sufficiently small. 
The inverse decoding problem is then given by the following operator equation:
\begin{equation}\label{opeq}\mathcal{F}(X)=\mathcal{M},\end{equation}
where $X$ denote the target objects that the individual player wants to decipher. In the context of optimal control, \eqref{opeq} can also be viewed as the inverse design of $X$ based on the expected output $\mathcal{M}$.

In a practical context, we aim to decode an MFG system with unknowns $X$ when the boundary of its state domain is accessible, yet the interior remains inaccessible. A viable approach involves waiting for the MFG system to attain its stationary state $U=(u_0, m_0)$, and then apply the boundary perturbations, namely $u_0|_{\Sigma\times[0,T]}+\epsilon \tilde{u}$, $m_0|_{\Sigma\times[0,T]}+\epsilon \tilde{m}$, to generate the dynamically perturbed Cauchy dataset on the boundary.  In fact, from a practical point of view, in order to unveil the underlying dynamics of an evolutionary process, a proper strategy is to generate dynamical perturbations around its stationary state and measure the corresponding perturbed data for the decoding and inversion processes. On the one hand, it is unobjectionable that there exist stationary states of generic dynamical processes since they assume the least energy. On the other hand, in order to retain the inherent structural information of the target dynamics, one should exert ``small" perturbations since large ones may generate unstable or chaotic dynamical behaviours. Yet, there are two more factors which will be crucial in deciding whether such a decoding strategy work or not. First, what type of dynamically perturbed data should be used for the decoding process and how to collect them? Second, are the data sufficient to recover the unknowns? We provide affirmative answers to these issues in the current article for decoding mean field games.

In fact, in Theorem 2.9 of \cite{LiuLoZhang2024decodingMFG} we recovered  $F$ with the Cauchy dataset $\mathcal{M}$. 
Here, the stable stationary solutions $U=(u_0,m_0)$ are fixed but unknown, and we can only measure the boundary data of their perturbed states $(u_{0,\epsilon}(x,t),m_{0,\epsilon}(x,t))$. We utilize such data to determine the dynamical function parameters $F$ and $u_T$, independent of the unknown fixed $U$. In the sense of \eqref{opeq}, this result in  Theorem 2.9 of \cite{LiuLoZhang2024decodingMFG} can be discussed as follows: 
Suppose there are two configurations $F_1$ and $F_2$, then the unique identifiability result can be written as 
\begin{equation}
    F_1=F_2 \quad \text{ if and only if }\mathcal{F}_{U}(F_1)=\mathcal{F}_{U}(F_2)
\end{equation}
for 
\begin{equation}
    \mathcal{F}_{U}(F)=\mathcal{M} \quad \text{ for an unknown but fixed }U.
\end{equation}

However, this is only a partial uniqueness result, with the a priori information that the stationary state $U$ is fixed. Yet, due to the nonlinearity of the system \eqref{eq:MFG0}, the system may possess more than one stationary state $U$. Hence, it will significantly benefit individual players if they can uniquely decipher the unknown stationary solution $U$ simultaneously, which also corresponds to the state as the time horizon extends towards infinity, i.e. $T\to\infty$.

Therefore, this study aims to tackle the outstanding issue presented in \cite{LiuLoZhang2024decodingMFG} to obtain the full unique identifiability result, by simultaneously recovering the stable stationary state $U$, in addition to the running cost $F$. Moreover, we made a further advancement to recover the Hamiltonian $\mathcal{H}$ that is more general than what was previously considered. Our approach involves utilizing the measured Cauchy dataset $\mathcal{M}$, i.e. 
\begin{equation}
    \mathcal{F}(U,F,\mathcal{H})=\mathcal{M}.
\end{equation}
In particular, we show the following unique identifiability result: For two configurations $(U_1,F_1,\mathcal{H}_1)$ and $(U_2,F_2,\mathcal{H}_2)$, it holds that 
\begin{equation}\label{formalresult}
U_1=U_2 \text{ and }F_1=F_2 \text{ and }\mathcal{H}_1=\mathcal{H}_2 \quad \text{ if and only if }\quad\mathcal{F}(U_1,F_1,\mathcal{H}_1)=\mathcal{F}(U_2,F_2,\mathcal{H}_2).\end{equation}
This comprehensive understanding of the population's and individual's general state will allow the player to anticipate the general behavior of the other players ($m_0$) in response to his/her actions ($u_0$), thereby empowering the agent to make more informed decisions when devising an optimal strategy. In this manner, we achieve a complete recovery of both the dynamical model ($F$ and $\mathcal{H}$) as well as the stationary states $U$. 
Further elaboration on these measurement maps and outcomes will be presented in the subsequent sections.

The major novelty that distinguishes our inverse problem study from most of the existing ones lies in the following two key aspects. Firstly, as in \cite{LiuLoZhang2024decodingMFG}, we consider MFG problems in a highly general form \eqref{eq:MFG0}, with a Hamiltonian that is more general than what was previously considered. We derive results pertaining to this form, thereby obtaining comprehensive analytical results that are applicable to this broad class of MFGs. This generality bolsters the practical significance of our discoveries and broadens their practical utility. 
Secondly, we have enhanced and refined the technique of linearization around the stationary states of the MFG system. As before, we assume that these stationary states are fixed but are not known a priori. Yet, in this work, we are able to fully derive these stationary states in the general time-dependent case, in addition to recovering the running cost $F$ and Hamiltonian $\mathcal{H}$. 
This technical advancement enables us to tackle the decoding problem in situations where the exact stationary states remain unknown, marking a significant breakthrough even within the broader domain of inverse problems for nonlinear PDEs. Consequently, this research lays the foundation for addressing inverse problems for nonlinear PDEs with general stationary states and furnishes a valuable tool for practical real-world applications.
Further discussion on these aspects will be elaborated in Section \ref{sect:technical}.

\subsection{Technical Developments and Discussion }\label{sect:technical}

Mean-field games (MFGs) is a powerful tool for analyzing intricate systems involving numerous rational decision-makers. Initially introduced in seminal works such as \cite{huang2006large, LasryLions1, LasryLions2}, MFGs offers a macroscopic perspective for various emergent phenomena and strategic interactions among large groups of individuals, in a wide spectrum of fields including transport, social sciences, finance and economics. This robust mathematical framework has therefore garnered significant attention and burgeoned into a vibrant research field.

The analysis of the well-posedness of the MFG system remains a key focus and an active area of research in diverse contexts. Noteworthy references such as \cite{ABC2017StationaryMFG,Ambrose2022existence,cardaliaguet2010notes,cardaliaguet2015weak,cardaliaguet2019master,cirant2020short,FerreiraGomesTada2019StationaryMFGDirichlet,ferreira2021existence} delve into this challenge, commonly referred to as the forward problem.

Given the broad spectrum of practical applications, exploring these real-world problems from an inverse perspective, such as the inverse design and decoding of MFGs, is both fascinating and practically significant. Decoding MFGs refers to the process of unraveling the underlying system dynamics and parameters of the mean-field game system from observed data. 
This pursuit holds pivotal implications across disciplines like economics, social sciences, and engineering, where uncovering the hidden structure of MFGs can furnish valuable insights and perspectives into the system dynamics and collective behaviors, empowering individuals to embrace more optimal strategies.

In this work, we focus on the recovery of both the intrinsic stable state $(u_0,m_0)$ of the MFG system \eqref{eq:MFG0}, as well as interaction factors such as the running cost $F$. These problems are known as the inverse problems for MFGs. Despite their practical importance and wide-ranging utility, such problems are far less studied in the literature. 
In \cite{LiuMouZhang2022InversePbMeanFieldGames,LiuZhang2022-InversePbMFG,LiuZhangMFG3,ding2023determining,LiuZhangMFG4,liulo2024determiningstatespaceanomalies,ding2024determininginternaltopologicalstructures}, MFG inverse problems were proposed and investigated with several novel unique identifiability results established, and in \cite{klibanov2023lipschitz,klibanov2023mean1,klibanov2023mean2,klibanov2023holder,liu2023stability,imanuvilov2023lipschitz1,imanuvilov2023unique,klibanov2023coefficient1,klibanov2023coefficient2,imanuvilov2023global} both uniqueness and stability results were derived in a variety of settings that required not only boundary measurements but also internal initial- and final-time measurements. In addition, we would also like to note that there are some numerical results \cite{chow2022numerical,ding2022mean,klibanov2023convexification} as well as some recent developments such as \cite{ren2023unique,ren2024reconstructing}. 
Notably, most of these works in \cite{ding2023determining,LiuMouZhang2022InversePbMeanFieldGames,LiuZhang2022-InversePbMFG,LiuZhangMFG3,ren2023unique,ren2024reconstructing,LiuZhangMFG4} employ the construction of complex geometric optics (CGO) solutions to establish uniqueness results. However, the construction of such CGO solutions becomes significantly more challenging when considering the probability density constraint on $m$ and the Neumann boundary conditions. Several attempts have been made to address this challenge, including \cite{LiuZhang2022-InversePbMFG,LiuZhangMFG3,LiuZhangMFG4,ren2024reconstructing}. 
In this work, we follow the approach of \cite{LiuLoZhang2024decodingMFG}, in that we do not impose any boundary conditions on the system \eqref{eq:MFG0}. Instead, we consider the set of all possible solutions $(u,m)$ to the system. This can be elucidated through practical scenarios: In a stock market where trading occurs via diverse platforms and locations, considering solely on online transactions from one country in a multi-national trading system may obviate the need for adhering strictly to probability density constraints and Neumann boundary conditions. Similar reasoning applies to multinational tech companies operating within a single nation or airport management scenarios involving varied airline operations. In this case, there are no probability density constraint and Neumann boundary conditions.

Indeed, the physical constraints associated with the MFG system \eqref{eq:MFG0}, including boundary conditions, positivity and probability density constraints, present intriguing and challenging technical novelties from a mathematical perspective. In addition, a major technical complexity is the nonlinear backward-forward coupling nature of the MFG system \eqref{eq:MFG0}. To address the nonlinearity, a robust strategy involving successive/high-order linearization is being used. Such a method has been extensively developed for various inverse problems related to nonlinear partial differential equations (PDEs), including in \cite{LinLiuLiuZhang2021-InversePbSemilinearParabolic-CGOSolnsSuccessiveLinearisation,LiuZhang2022-InversePbMFG,LiuZhangMFG3,LiuZhangMFG4,liu2023determining,li2023inverse,li2024inverse,LiLoCAC2024}. Among several technical advancements, a key idea in \cite{LiuLoZhang2024decodingMFG} is to linearize around any solution, including non-trivial ones that are not known a priori. This means that the linearized systems remain coupled in nature, unlike the approach in \cite{LiuMouZhang2022InversePbMeanFieldGames} where the linearization is conducted around the trivial solution $(0,0)$ allowing for decoupling. Yet, in \cite{LiuLoZhang2024decodingMFG}, the stationary solution was only derived in the case of the stationary MFG. In this work, we treat the more general time-dependent case. We perform linearization around any unknown stable stationary solution, successfully recovering not only the higher-order Taylor coefficients of $F$, but also the unknown stable stationary solution in the process. This in particular, extends the work of \cite{LiuLoZhang2024decodingMFG}. 

The challenges highlighted are not exclusive to MFG inverse problems. Similar constraints and hurdles manifest in a variety of coupled or uncoupled PDE systems, as exemplified in \cite{LinLiuLiuZhang2021-InversePbSemilinearParabolic-CGOSolnsSuccessiveLinearisation,liu2023determining,li2023inverse,LiLoCAC2024,li2024inverse,DingLiuLo2024inverseproblemscouplednonlocal}. In these works, high-order variation and high-order/successive linearization schemes around various trivial and non-trivial solutions have also been applied to investigate the associated inverse boundary problems. We believe that the mathematical techniques and strategies developed in this work can be applied to these innovative and compelling contexts, enabling the recovery of the unknown stationary state of the system in addition to function parameters in the system. This offers novel perspectives on inverse boundary problems for nonlinear PDEs, with the potential to yield theoretical and practical outcomes of significant importance.

The subsequent sections of this manuscript are structured as outlined below. Section \ref{sect:prelim} presents some preliminary results, establishing the primary framework and stating the main theorem. In Section \ref{sec:wellpose}, we briefly recall the local well-posedness of the forward MFG system. In Section \ref{sec:linearzation_discussion}, we detail the high-order linearization method, which is necessary for the proof of the main result. The proof of the main theorem is finally presented in Section \ref{sec:mainproof}.

\section{Preliminaries}\label{sect:prelim}

\subsection{Admissible class}\label{sec:admissible}
The H\"older space $C^{k+\alpha}(\overline{\Omega})$, $0<\alpha<1$, $k\in\mathbb{N}$, is defined to be the subspace of $C^{k}(\overline{\Omega})$, such that a function lies in $C^{k+\alpha}(\overline{\Omega})$ if and only if, for all $l=(l_1,l_2,\ldots,l_n)\in \mathbb{N}^n$ with $|l|\leq k$, the derivatives $D^l:=\partial_{x_1}^{l_1}\partial_{x_2}^{l_2}\cdots\partial_{x_n}^{l_n}$ exist and are H\"older continuous with exponent $\alpha$. The corresponding norm is given by
	\begin{equation}
		\|\psi\|_{C^{k+\alpha}(\overline{\Omega}) }:=\sum_{|l|=k}\sup_{x\neq y}\frac{|D^l\psi(y)-D^l\psi(x)|}{|y-x|^{\alpha}}+\sum_{|l|\leq k}\|D^l\psi\|_{L^\infty(\Omega)}.
	\end{equation}
For functions depending on both time and space variables, we define the space $C^{k+\alpha, \frac{k+\alpha}{2}}(Q)$ if the derivatives $D^lD^{j}_t$ exist and are H\"older continuous with exponent $\alpha$ in $x$ and $\frac{k+\alpha}{2} $ in $t$ for all  $l\in \mathbb{N}^n$, $j\in\mathbb{N}$ with $|l|+2j\leq k.$ The norm is defined as
	\begin{equation}
		\begin{aligned}
			\|\psi\|_{ C^{k+\alpha, \frac{k+\alpha}{2}}(Q)}:&=\sum_{|l|+2j= k}\sup_{t,x\neq y}\frac{|\psi(y,t)-\psi(x,t)|}{|y-x|^{\alpha}}\\
			&+\sum_{|l|+2j= k}\sup_{t\neq t',x} \frac{|\psi(x,t')-\psi(x,t)|}{|t'-t|^{\alpha/2}}+\sum_{|l|+2j\leq k}\|D^lD^j_t\psi\|_{L^\infty(Q)}.
		\end{aligned}
	\end{equation}
    
Now we establish sufficient conditions to guarantee that the uniqueness result \eqref{formalresult} holds. 
As discussed in the introduction, the Hamiltonian $\mathcal{H}$ in \eqref{eq:MFG0} adopts a convex nonlinear form. In this work, we assume that the $\mathcal{H}$ is a quadratic form, i.e. for a smooth Riemannian metric $A=(g_{kj}(x))_{k,j=1}^n$, the Hamiltonian $\mathcal{H}$ defined in the phase space is of the form
\begin{equation}\label{eq:ham1}
\mathcal{H}(x,p)=\sum_{kj}g_{kj}p_kp_j,
\end{equation}
where $p=\nabla u(x, t)$ is the momentum variable. This signifies that the energy Lagrangian of the MFG system is kinetic, which is often found in practical situations (cf. \cite{LiuMouZhang2022InversePbMeanFieldGames,ding2022mean} ). 
\begin{definition}\label{def:A}
    A Riemannian metric $A$ is called admissible if it belongs to some fixed known conformal class $C_g$, i.e. $A(x)=\kappa(x)g(x)$ for $\kappa(x)\in C^\infty(\Omega)$ and $g(x)$ is a known Riemannian metric.
\end{definition}
Next, we introduce the admissible class of $F$, in a similar way as in \cite{LiuMouZhang2022InversePbMeanFieldGames,LiuZhangMFG3,LiuZhangMFG4,LiuLoZhang2024decodingMFG}. For the completeness of this paper, we list it here.
\begin{definition}\label{AdmissClass1}
	We say that  $Y(x,z):\mathbb{R}^n\times\mathbb{C}\to\mathbb{C}$ is admissible, denoted by $Y\in\mathcal{A}$, if the following two conditions are satisfied:
	\begin{enumerate}[label=(\roman*)]
		\item The map $z\mapsto Y(\cdot,z)$ is holomorphic with value in $C^{2+\alpha}(\mathbb{R}^n)$ for some $\alpha\in(0,1)$;
			\item $Y(x,m_0(x))=0$ for all $x\in\mathbb{R}^n$ and given $m_0(x)$ depending only on $x$;
            \item $Y^{(1)}(x)=0$ for all $x\in \Omega$.
		\end{enumerate}

		Clearly, if the conditions (i) and (ii) are fulfilled, $A$ can be expanded into a power series as follows:
		\begin{equation}\label{eq:G}
			Y(x,z)=\sum_{k=1}^{\infty} Y^{(k)}(x)\frac{(z-m_0)^k}{k!},
		\end{equation}
		where $ Y^{(k)}(x)=\frac{\partial^kY}{\partial z^k}(x,m_0)\in C^{2+\alpha}(\mathbb{R}^n).$
\end{definition}

The admissibility conditions in Definition \ref{AdmissClass1} shall be imposed as a priori conditions on $F$, as in \cite{LiuLoZhang2024decodingMFG}. 
These assumptions imply that $F$ takes a polynomial form, a common occurrence in real-world scenarios. Specifically, Assumption (iii) stipulates that $F$ is minimally of quadratic order, a characteristic often observed in practical settings where players mimic others (see, for instance, \cite{BardiQuadraticCostMFG,HuangHuangQuadraticCostMFG,BausoPesentiQuadraticCostMFG,USGQuadraticCostMFG}). Hence, from a practical point of view, these analytic assumptions on $F$ are unobjectionable.

\subsection{Main unique identifiability results}

Having established these definitions, we state our main result, whereby we fully recover the running cost as well as the stationary state of the quadratic MFG system
\begin{equation}\label{eq:MFG2}
    \begin{cases}
        -\partial_t u(x,t) -\Delta u(x,t) + [\nabla u(x,t)]^TA(x)\nabla u(x,t) = F(x,m) &\quad \text{in }Q,\\
        \partial_t m(x,t) -\Delta m(x,t) - 2\nabla\cdot\left(m(x,t)[\nabla u(x,t)]^TA(x)\right) = 0  &\quad \text{in }Q,\\
        u(x,T)=u_T,\quad m(x,0)=f(x) &\quad \text{in }\Omega,
    \end{cases}
\end{equation} 
for bounded positive real function $(x,t)>0$, $(x,t)\in C^{1,0}(\tilde{Q})$.

Consider the measurement map 
$\mathcal{M}_{F,U,A}$ given by \[\mathcal{M}_{F,U,A}=\left.\left(u,\nabla u,m,\nabla m\right)\right|_{\Gamma}\to F,U,A \]
associated to \eqref{eq:MFG2}, where $U=(u_0,m_0)$ and $\Gamma:=\Sigma\times(0,T)$. 
First, we define
\begin{definition}\label{stable}
Let $F\in\mathcal{A}$. Let $g,h\in C^{2+\alpha,1+\frac{\alpha}{2}}(\Gamma)$ satisfy the following  compatibility conditions:
\begin{equation}\label{compatibility conditions }
\begin{cases}
     g(x,T)=0&\quad \text{ in }\Sigma\\
      -\partial_t g -\Delta g + [\nabla u_0]^TA\nabla g+ [\nabla g]^TA\nabla u_0=0 &\quad \text{ in }\Sigma\\
     h(x,0)=0&\quad \text{ in }\Sigma\\
     \partial_t h -\Delta h - 2\nabla\cdot\left( m_0[\nabla g]^TA\right) -2\nabla \cdot \left( h[\nabla u_0]^TA\right)= 0 &\quad \text{ in }\Sigma\\
\end{cases}
\end{equation}

    A solution $U=(u_0,m_0)$ is called a \emph{stable solution} of \eqref{eq:MFG2} if for any
    $g,h$ that satisfy the compatibility conditions, the system
  \begin{equation}\label{MFG2:stable}
    \begin{cases}
        -\partial_t u^{(1)}(x,t) -\Delta u^{(1)}(x,t) + [\nabla u_0(x)]^TA(x) \nabla u^{(1)}(x,t)\\\quad\quad\quad\quad\quad\quad\quad\quad\quad\quad\quad\quad + [\nabla u^{(1)}(x,t)]^TA(x) \nabla u_0(x)= 0 &\quad \text{in }Q,\\
        \partial_t m^{(1)}(x,t) -\Delta m^{(1)}(x,t) - 2\nabla\cdot\left( m_0(x)[\nabla u^{(1)}(x,t)]^TA(x)\right) \\\quad\quad\quad\quad\quad\quad\quad\quad\quad\quad\quad\quad -2\nabla \cdot \left( m^{(1)}(x,t)[\nabla u_0(x)]^TA(x)\right)= 0  &\quad \text{in }Q,\\
        u^{(1)}(x,t)=g,\quad m^{(1)}(x,t)=h&\quad \text{in }\Gamma,\\
        u^{(1)}(x,T)=0, \quad m^{(1)}(x,0) = 0  &\quad \text{in }\Omega.
    \end{cases}
\end{equation}
admits a unique solution $(u,m)\in [ C^{2+\alpha,1+\frac{\alpha}{2}}(Q)]^2.$
\end{definition}

\begin{definition}
    A solution $U=(u_0,m_0)$ is a \emph{stationary solution} of \eqref{eq:MFG2} if it satisfies the following system
    \begin{equation}\label{MFG2Linear0}
    \begin{cases}
        -\Delta u_0(x) + [\nabla u_0(x)]^TA(x) \nabla u_0(x)= 0&\quad \text{in }Q,\\
        -\Delta m_0(x) - 2\nabla\cdot\left( m_0(x)[\nabla u_0(x)]^TA(x)\right)= 0  &\quad \text{in }Q.
    \end{cases}
\end{equation}
\end{definition}

    Then, we are able to reconstruct the running costs, Hamiltonian and stable stationary state solutions, through measurements for general analytic cost functions, and the following uniqueness result holds:

\begin{theorem}\label{mainthm}
For $i=1,2$, let $A_i$ be an admissible Riemannian metric and $F_i\in\mathcal{A}$ such that $U_i=(u_{0,i},m_{0,i})\in [C^{2+\alpha}(\Omega)]^2$ is a stable stationary solution of the following system:
   \begin{equation}\label{eq:MFG2i}
    \begin{cases}
        -\partial_t u(x,t) -\Delta u(x,t) + [\nabla u(x,t)]^TA_i(x)\nabla u(x,t) = F_i(x,m) &\quad \text{in }Q,\\
        \partial_t m(x,t) -\Delta m(x,t) - 2\nabla\cdot\left(m(x,t)[\nabla u(x,t)]^TA_i(x)\right) = 0  &\quad \text{in }Q,\\
        u(x,T)=u_{T,i},\quad m(x,0)=f_i(x) &\quad \text{in }\Omega,
    \end{cases}
\end{equation}
such that the Hamiltonian $\mathcal{H}_i=[\nabla u]^TA_i\nabla u$ is non-degenerate at the stationary state $U_i$. 
 Let $\mathcal{M}_{F_i,A_i,U_i}$ be the associated measurements. Suppose \[\mathcal{M}_{F_1,A_1,U_1}=\mathcal{M}_{F_2,A_2,U_2},\] 
     Then, one has \[F_1=F_2,\quad \quad U_1=U_2\]
     and 
     \[A_1(x)=A_2(x)\quad \text{ in }Q,\]
     up to the conformal class $C_g$.
\end{theorem}


As discussed in the introduction, we consider a stable stationary state $U$ of the time-dependent MFG system \eqref{eq:MFG2}, and measure the time-dependent Cauchy dataset of its different perturbations. However, the stable stationary state $U$ remains unknown. Our main novelty in this work, as compared to \cite{LiuLoZhang2024decodingMFG}, is the unique identifiability of this stable stationary solutions $U$ and Hamiltonian $\mathcal{H}$, in addition to the running costs $F$. Such a result has only previously been obtained in the stationary case, and in this work we derive the result for the general time-dependent case.

\section{Well-posedness of the forward problems}\label{sec:wellpose}
Before we begin our study, we first derive the well-posedness of the associated forward problems. This is necessary because the regularity of the solutions to the forward problems is crucial to our inverse problem study. We will be using the high order linearization method to treat the MFG system, and this necessitates the infinite differentiability of the system with respect to small variations around a given solution.

\begin{theorem}\label{wellpose}
    Let $F\in\mathcal{A}$ and suppose $(u_0,m_0)$ is a stable solution of \eqref{eq:MFG2}. Then,
    \begin{enumerate}[label=(\alph*)]	
		\item
		there exist constants $\delta>0$ and $C>0$ such that for any 
		\begin{equation*}
		\begin{aligned}
		 &g\in B_{\delta}( C^{2+\alpha,1+\frac{\alpha}{2}}(\Gamma)) :=\{f\in  C^{2+\alpha,1+\frac{\alpha}{2}}(\Gamma): \|f\|_{ C^{2+\alpha,1+\frac{\alpha}{2}}(\Gamma)}\leq\delta,(f(x,0)+m_0)|_{\Sigma}=0\},\\
          &h\in D_{\delta}( C^{2+\alpha,1+\frac{\alpha}{2}}(\Gamma)) :=\{f\in  C^{2+\alpha,1+\frac{\alpha}{2}}(\Gamma): \|f\|_{ C^{2+\alpha,1+\frac{\alpha}{2}}(\Gamma)}\leq\delta,(f(x,0)+v_0)|_{\Sigma}=0\},
		\end{aligned}    
		\end{equation*}

		the system 
  \begin{equation}\label{eq:MFG_wellpose}
    \begin{cases}
        -\partial_t u(x,t) -\Delta u(x,t) + [\nabla u(x,t)]^TA(x)\nabla u(x,t) = F(x,m) &\quad \text{in }Q,\\
        \partial_t m(x,t) -\Delta m(x,t) - 2\nabla\cdot\left( m(x,t)[\nabla u(x,t)]^TA(x)\right) = 0  &\quad \text{in }Q,\\
        u(x,t)=u_0|_{\Gamma  }+h,\quad m(x,t)=m_0|_{\Gamma}+g&\quad \text{in }\Gamma,\\
        u(x,T)=u_T,\quad m(x,0)=f(x) &\quad \text{in }\Omega.
    \end{cases}
\end{equation} 

  has a solution $(u,m)\in
		[C^{2+\alpha,1+\frac{\alpha}{2}}(Q)]^2$ which satisfies
		\begin{equation}
  \begin{aligned}
      	\|(u-u_0,m-m_0)\|_{ C^{2+\alpha,1+\frac{\alpha}{2}}(Q)}:&= \|u-u_0\|_{C^{2+\alpha,1+\frac{\alpha}{2}}(Q)}+ \|m-m_0\|_{C^{2+\alpha,1+\frac{\alpha}{2}}(Q)}\\  
    &\leq C(\|g\|_{ C^{2+\alpha,1+\frac{\alpha}{2}}(\Gamma)}+\|h\|_{  C^{2+\alpha,1+\frac{\alpha}{2}}(\Gamma)}).
  \end{aligned}
		\end{equation}
		Furthermore, the solution $(u,m)$ is unique within the class
		\begin{equation}
      \{ (u,m)\in  C^{2+\alpha,1+\frac{\alpha}{2}}(Q)\times C^{2+\alpha,1+\frac{\alpha}{2}}(Q): \|(u-u_0,m-m_0)\|_{ C^{2+\alpha,1+\frac{\alpha}{2}}(Q)}\leq C\delta \}.
		\end{equation}		
		\item Define a function 
		\[
		S: B_{\delta}( C^{2+\alpha,1+\frac{\alpha}{2}}(\Gamma)\times D_{\delta}( C^{2+\alpha,1+\frac{\alpha}{2}}(\Gamma)\to C^{2+\alpha,1+\frac{\alpha}{2}}(Q)\times C^{2+\alpha,1+\frac{\alpha}{2}}(Q)
		\] by \[S(g,h):=(u,m),\]
		where $(u,m)$ is the unique solution to the MFG system \eqref{eq:MFG_wellpose}.
		Then for any $(g,h)\in B_{\delta}( C^{2+\alpha,1+\frac{\alpha}{2}}(\Gamma))^2$, $S$ is holomorphic at $(g,h)$.
	\end{enumerate}
\end{theorem}

\begin{proof}
	Let 
	\begin{align*}
        &X_0:=\{ f\in C^{2+\alpha}(\Omega): (f(x,0)+m_0(x,0))|_{\Sigma}=0 \}\\
        &X_0':=\{ f\in C^{2+\alpha}(\Omega): (f(x,T)+u_0(x,T))|_{\Sigma}=0  \}\\
        &X_1:=\{f\in  C^{2+\alpha,1+\frac{\alpha}{2}}(\Gamma ):  ((m_0+f)(x,0))|_{\Sigma  }=0\}\\
		&X_1':= \{ f\in  C^{2+\alpha,1+\frac{\alpha}{2}}(\Gamma ):  ((u_0+f)(x,T)|_{\Sigma})=0 \} , \\
        &X_2:=\{f\in  C^{2+\alpha,1+\frac{\alpha}{2}}(Q):f(x,0)_{\Sigma}=0,(-\partial_t f(x,0)-\Delta f(x,0))|_{\Sigma}=0 \}\\
        &X_2':=\{f\in  C^{2+\alpha,1+\frac{\alpha}{2}}(Q):f(x,T)_{\Sigma}=0,(-\partial_t f(x,T)+\Delta f(x,T)|_{\Sigma}=0 \}\\
		&X_3:=X_0\times X_0'\times [ \{ h\in C^{\alpha,\frac{\alpha}{2}}(Q):h(x,0)|_{\Sigma}=0\}]^2,
	\end{align*} and we define a map $\mathscr{L}:X_1\times X'_1\times X_2\times X_2' \to X_3$ by that for any $(g,h,\tilde u,\tilde m)\in X_1\times X'_1\times X_2\times X_2'$,
	\begin{align*}
		\mathscr{L}( g,h,\tilde u,\tilde m)(x,t)
		:=&\big( \tilde u(x,t)|_{\Gamma}-u_0|_{\Gamma}-g , \tilde m(x,t)|_{\Gamma}-m_0|_{\Gamma}-h , -\partial_t\tilde u(x,t)
		\\ &-\Delta \tilde u(x,t)+[\nabla \tilde u(x,t)]^TA(x)\nabla \tilde u(x,t)- F(x,t,\tilde m(x,t)), 
		\\ &\partial_t \tilde m(x,t)-\Delta \tilde m(x,t)-2\nabla\cdot(\tilde m(x,t)[\nabla \tilde u(x,t)]^TA(x))  \big) .
	\end{align*}

	First, we show that $\mathscr{L} $ is well-defined. Since the
	H\"older space is an algebra under the point-wise multiplication, we have $[\nabla u]^TA\nabla u, \nabla\cdot( m[\nabla u]^TA)  \in C^{\alpha,\frac{\alpha}{2}}(Q ).$
	By the Cauchy integral formula,
	\begin{equation}\label{eq:F1}
		F^{(k)}\leq \frac{k!}{R^k}\sup_{|z|=R}\|F(\cdot,\cdot,z)\|_{C^{\alpha,\frac{\alpha}{2}}(Q ) },\ \ R>0.
	\end{equation}
	Then there is $L>0$ such that for all $k\in\mathbb{N}$,
	\begin{equation}\label{eq:F2}
		\left\|\frac{F^{(k)}}{k!}m^k\right\|_{C^{\alpha,\frac{\alpha}{2}}(Q )}\leq \frac{L^k}{R^k}\|m\|^k_{C^{\alpha,\frac{\alpha}{2}}(Q)}\sup_{|z|=R}\|F(\cdot,\cdot,z)\|_{C^{\alpha,\frac{\alpha}{2}}(Q) }.
	\end{equation}
	By choosing $R\in\mathbb{R}_+$ large enough and by virtue of \eqref{eq:F1} and \eqref{eq:F2}, it can be seen that the series  converges in $C^{\alpha,\frac{\alpha}{2}}(Q )$ and therefore $F(x,t,m(x,t))\in  C^{\alpha,\frac{\alpha}{2}}(Q).$  This implies that $\mathscr{L} $ is well-defined. 
    
    Now we move to show that $\mathscr{L}$ is holomorphic. It suffices to verify that it is weakly holomorphic because  $\mathscr{L}$ is clearly locally bounded. In other words, we aim to show that the map
	$$\lambda\in\mathbb C \mapsto \mathscr{L}((g,h,\tilde u,\tilde m)+\lambda (\bar g,\bar h,\bar u,\bar m))\in X_3,\quad\text{for any $(\bar g,\bar h,\bar u,\bar m)\in X_1\times X_1'\times X_2\times X_2'$}$$
	is holomorphic. In fact, this follows from the condition that $F\in\mathcal{A}$ and $G\in\mathcal{B}$.

	Note that $ \mathscr{K}( 0,0,u_0,m_0)= 0$ because $(u_0,m_0)$ is a stable solution, so we have that $\nabla_{(\tilde u,\tilde m)} \mathscr{K} (0,0,u_0,m_0)$ is a linear isomorphism between $X_2\times X_2'$ and $X_3$. Hence, by the implicit function theorem, this theorem is proved.
\end{proof}

Furthermore, the maps of boundary data to solution are $C^{\infty}$-Fr\'{e}chet differentiable,
thus we can also derive that the corresponding Dirichlet-to-Neumann
 map is also $C^{\infty}$-Fr\'{e}chet differentiable.

\section{High order linearization}\label{sec:linearzation_discussion}
\subsection{Linearization in probability space}

Recall that the population distribution in \eqref{eq:MFG0} is denoted by the function $m$, which lies in the space $\mathcal{P}(\Omega) $, the set of probability measures on $\Omega$.
	
	\begin{definition}\label{def_der_1}
		Let $B :\mathcal{P}(\Omega)\to\mathbb{R}$ be a real function. We say that $B$ is of class $C^1$ if there exists a continuous map $K:  \mathcal{P}(\Omega)\times \Omega\to\mathbb{R}$ such that, for all $m_1,m_2\in\mathcal{P}(\Omega) $ we have
		\begin{equation}\label{derivation}
			\lim\limits_{s\to 0^+}\frac{B\big(m_1+s(m_2-m_1)-B(m_1)\big)}{s}=\int_{\Omega} K(m_1,x)d(m_2-m_1)(x).
		\end{equation}
	\end{definition}
Here, $K$ is defined up to additive constants. The derivative
	$\dfrac{\delta B}{\delta m}$ is then defined by the unique map $K$ satisfying $\eqref{derivation}$ such that 
	\begin{equation}
		\int_{\Omega} K(m,x) dm(x)=0.
	\end{equation}
	In a similar way, we can define the higher order derivatives of $B$ (see \cite{ricciardi2022master} for related discussions).
	The Wasserstein distance between $m_1$ and $m_2$ in $\mathcal{P}(\Omega)$ is defined as follows:
	\begin{definition}\label{W_distance}
		Let $m_1,m_2\in\mathcal{P}(\Omega)$. Define
		\begin{equation}
			d_1(m_1,m_2):=\sup_{Lip(\psi)\leq 1}\int_{\Omega}\psi(x)d(m_1-m_2)(x),
		\end{equation}
		where $Lip(\psi)$ denotes the Lipschitz constant for a Lipschitz function, i.e., 
		\begin{equation}\label{eq:Lip1}
			Lip(\psi)=\sup_{x, y\in\Omega, x\neq y}\frac{|\psi(y)-\psi(x)|}{|y-x|}. 
		\end{equation}
	\end{definition}
	
		We need to point out that in Definitions~\ref{def_der_1} and \ref{W_distance}, $m$ (i.e. $m_1$ or $m_2$) is viewed as a distribution. However, in other parts of the paper, we use $m$ to denote the density of a distribution such as in the MFG system \eqref{eq:MFG0}.

Finally, the relationship between the MFG system and its linearized system can be stated as follows:
 Let $(u_1, m_1)$ and $(u_2, m_2)$ be two solutions of the MFG system, associated with the starting initial conditions $m_{0}^1$
 and $m_{0}^2$. Let $(s,\rho)$ be the solution of the linearized system  related to $(u_2, m_2)$, with initial condition $m_{0}^1-m_{0}^2$. Then we have 
the norms of $u_1-u_2-s$ and $u_1-u_2-\rho$ in suitable function spaces are bounded by 
$Cd_1(m^1_{0},m_0^2)$. For details of this theorem we also refer to \cite{ricciardi2022master}. However, since we work in H\"{o}lder spaces in this paper and assume the measure $m$  always belongs to $C^{2+\alpha}(\Omega)$ at least, this result is replaced by Theorem $\ref{wellpose}$.

Next, we write out the linearized systems corresponding to our system \eqref{eq:MFG2} in the next subsection.

\subsection{High order linearized systems}
Since the local well-posedness is known from Theorem \ref{wellpose}, we are ready to conduct the high order linearization. 
Let 
\begin{equation}
\begin{aligned}
      &u(x,t)|_{\Gamma}=u_0|_{\Gamma}+\sum_{l=1}^{N}\varepsilon_lg_l|_{\Gamma}\\
      &m(x,t)|_{\Gamma}=m_0|_{\Gamma}+\sum_{l=1}^{N}\varepsilon_lh_l|_{\Gamma}
\end{aligned}
\end{equation}
where $g_l,h_l\in C^{2+\alpha,1+\frac{\alpha}{2}}(\mathbb{R}^n\times\mathbb{R})$  and $\varepsilon = (\varepsilon_1,\varepsilon_2, \dots,\varepsilon_N) \in \mathbb{R}^N$ with $|\varepsilon| = \sum_{l=1}^N |\varepsilon_l|$ small enough. Recall that $(u_0,m_0)$ is a stationary stable solution of system \eqref{eq:MFG2} and $m_0\geq 0$, so $m(x,t)|_{\Gamma}\geq0$ for $\varepsilon$ small enough. Then, by Theorem \ref{wellpose}, there exists a unique solution $(u(x,t;\varepsilon), m(x,t;\varepsilon))$ of \eqref{eq:MFG2} and $(u(x,t;0), m(x,t;0))=(u_0,m_0)$ is the solution of \eqref{eq:MFG2} when $\varepsilon = 0$. Therefore, the zeroth-order linearization system is given by:
\begin{equation}\label{MFG2Linear0Order}
    \begin{cases}
        -\Delta u_0(x) + [\nabla u_0(x)]^TA(x)[\nabla u_0(x)]= 0&\quad \text{in }Q,\\
        -\Delta m_0(x) - 2\nabla\cdot\left( m_0(x)[\nabla u_0(x)]^TA(x)\right)= 0  &\quad \text{in }Q.
    \end{cases}
\end{equation}

Let $S$ be the solution operator of \eqref{eq:MFG2} defined in Theorem \ref{wellpose}. Then there exists a bounded linear operator $\mathcal{B}$ from $\mathcal{H}:=B_{\delta}( C^{2+\alpha,1+\frac{\alpha}{2}}(\Gamma))\times D_{\delta}( C^{2+\alpha,1+\frac{\alpha}{2}}(\Gamma))$ to $[C^{2+\alpha,1+\frac{\alpha}{2}}(Q)]^2$ such that
\begin{equation}
	\lim\limits_{\|(g,h)\|_{\mathcal{H}}\to0}\frac{\|S(g,h)-S(u_0,m_0)- \mathcal{B}(g,h)\|_{[C^{2+\alpha,1+\frac{\alpha}{2}}(Q)]^2}}{\|(g,h)\|_{\mathcal{H}}}=0,
\end{equation} 
where $\|(g,h)\|_{\mathcal{H}}:=\|g\|_{ C^{2+\alpha,1+\frac{\alpha}{2}}(\Gamma)      }+\|h\|_{C^{2+\alpha,1+\frac{\alpha}{2}}(\Gamma)}$.
Now we consider $\varepsilon_l=0$ for $l=2,\dots,N$ and fix $f_1$.
Notice that if $F\in\mathcal{A}$ , then $F$ depends on the distribution $m$ locally. We have that 
	$$
	\dfrac{\delta F}{ \delta m}(x,m_0)(\rho(x,t)):=\left<\dfrac{\delta F}{ \delta m}(x,m_0,\cdot),\rho(x,t)\right>_{L^2}=
	F^{(1)}(x)\rho(x,t),
	 $$ 
	 up to a constant. 
Then it is easy to check that $\mathcal{B}(g,h)\Big|_{\varepsilon_1=0}$ is the solution map of the following system which is called the first-order linearization system:
\begin{equation}\label{MFG2Linear1}
    \begin{cases}
        -\partial_t u^{(1)}(x,t) -\Delta u^{(1)}(x,t) + [\nabla u_0(x)]^TA(x)\nabla u^{(1)}(x,t) \\\quad\quad\quad\quad\quad\quad\quad\quad\quad\quad\quad\quad\quad\quad+ [\nabla u^{(1)}(x,t)]^TA(x)\nabla u_0(x)= 0 &\quad \text{in }Q,\\
        \partial_t m^{(1)}(x,t) -\Delta m^{(1)}(x,t) - 2\nabla\cdot\left( m_0(x)[\nabla u^{(1)}(x,t)]^TA(x)\right) \\\quad\quad\quad\quad\quad\quad\quad\quad\quad\quad\quad\quad\quad\quad-2\nabla \cdot \left( m^{(1)}(x,t)[\nabla u_0(x)]^TA(x)\right)= 0  &\quad \text{in }Q,\\
        u^{(1)}(x,t)=g_1,\quad m^{(1)}(x,t)=h_1 &\quad \text{in }\Gamma,\\
        u^{(1)}(x,T)=0, \quad m^{(1)}(x,0) = 0  &\quad \text{in }\Omega.
    \end{cases}
\end{equation}

 In the following, we define
\begin{equation}\label{eq:ld1}
 (u^{(1)}, m^{(1)} ):=\mathcal{B}(g,h)\Big|_{\varepsilon_1=0}. 
 \end{equation}
For notational convenience, we write
\begin{equation}\label{eq:ld2}
u^{(1)}=\partial_{\varepsilon_1}u(x,t;\varepsilon)|_{\varepsilon=0}\quad\text{and}\quad m^{(1)}=\partial_{\varepsilon_1}m(x,t;\varepsilon)|_{\varepsilon=0}.
\end{equation}
We employ these notations in our following discussion to simplify the exposition and their meaning should be clear from the context. In a similar manner, we can define, for all $l\in \mathbb{N}$, $u^{(l)} := \left. \partial_{\varepsilon_l} u \right\rvert_{\varepsilon = 0}$, $m^{(l)} := \left. \partial_{\varepsilon_l} m \right\rvert_{\varepsilon = 0}$, 
we obtain a sequence of similar systems.

For the higher orders, we consider
\[u^{(1,2)}:= \left. \partial_{\varepsilon_1} \partial_{\varepsilon_2} u \right\rvert_{\varepsilon = 0}, \quad m^{(1,2)}:= \left. \partial_{\varepsilon_1} \partial_{\varepsilon_2} m \right\rvert_{\varepsilon = 0}.\]
Similarly, $(u^{(1,2)},m^{(1,2)})$ can be viewed as the output of the second-order Fr\'echet derivative of $S$ at a specific point. By following similar calculations in deriving \eqref{MFG2Linear1}, one can show that the second-order linearization is given as follows:
\begin{equation}\label{MFGQuadraticLinear2}
    \begin{cases}
        -\partial_t u^{(1,2)} -\Delta u^{(1,2)} + [\nabla u_0 ]^TA\nabla u^{(1,2)} + [\nabla u^{(2)}]^TA \nabla u^{(1)} + [\nabla u^{(1)}]^TA \nabla u^{(2)}\\\qquad\qquad\qquad\qquad\qquad\qquad +  [\nabla u^{(1,2)} ]^TA\nabla u_0 = F^{(2)}m^{(1)}m^{(2)} &\quad \text{in }Q,\\
        \partial_t m^{(1,2)} -\Delta m^{(1,2)} - 2\nabla\cdot\left( m_0[\nabla u^{(1,2)}]^TA\right) - 2\nabla \cdot\left( m^{(2)} [\nabla u^{(1)}]^TA\right)  \\\qquad\qquad\qquad\qquad\qquad\qquad -2\nabla\cdot\left( m^{(1)} [\nabla u^{(2)}]^TA\right) - 2\nabla\cdot\left( m^{(1,2)} [\nabla u_0]^TA\right) = 0   &\quad \text{in }Q,\\
        u^{(1,2)}(x,T)=0,\quad m^{(1,2)}(x,0) = 0  &\quad \text{in }\Omega.
    \end{cases}
\end{equation}

Inductively, for $N\in \mathbb{N}$, we consider
\[u^{(1,2,\dots,N)}:= \left. \partial_{\varepsilon_1} \partial_{\varepsilon_2} \cdots \partial_{\varepsilon_N} u \right\rvert_{\varepsilon = 0}, \quad m^{(1,2,\dots,N)}:= \left. \partial_{\varepsilon_1} \partial_{\varepsilon_2} \cdots \partial_{\varepsilon_N} m \right\rvert_{\varepsilon = 0},\] and obtain a sequence of parabolic systems.

\section{Proof of Main Theorem}\label{sec:mainproof}
Before we begin the proof of the main theorem, we first give an essential auxiliary lemma regarding the construction of complex geometric optics solutions, which is a result of \cite{ParabolicConvectionVectorCGO}.

\begin{theorem}\label{thm:cgo} Let $\psi:=\rho^2t+\rho\zeta\cdot x$ for some fixed $\zeta\in\mathbb{S}^{n-1}$. Suppose $\phi\in [C^1(Q)]^n$ and $q\in L^{\infty}(Q)$. For $\xi\in\mathbb{R}^n$ and $\tau\in\mathbb{R}$, and any arbitrary $\chi\in C_c^\infty((0,T))$ and $\rho>0$,
	\begin{enumerate}
	    \item there exist solutions $w\in H^1(0,T;H^{-1}(\Omega))\cap L^2(0,T;H^1(\Omega))$ to the forward parabolic equation \begin{equation}\label{forwardcgoeq}\begin{cases}
	\partial_t w - \Delta w -\phi\cdot\nabla w+ qw=0\quad &\text{ in }Q,\\
 w(x,t)=0 &\text{ on }\Sigma,\\
 w(x,0)=0 &\text{ in }\Omega,\end{cases}
\end{equation}  of the form
	\begin{equation}\label{cgo1}
		w=e^\psi\left(\chi(t)e^{-i(x,t)\cdot(\xi,\tau)}\exp\left(\frac{1}{2}\int_0^\infty \zeta\cdot\phi(x+s\xi)\,ds\right) +z_+(x,t)\right)
	\end{equation}
	where $z_+(x,t;\phi,q)\in H^1(0,T;H^{-1}(\Omega))\cap L^2(0,T;H^1(\Omega))$ satisfies the following condition:
	\begin{equation}\label{decay1}
		\lim_{\rho\to\infty}\norm{z_+(x,t;\phi,q)}_{L^2(Q)}=0.
	\end{equation}
 
 \item There exist solutions $v\in H^1(0,T;H^{-1}(\Omega))\cap L^2(0,T;H^1(\Omega))$ to the backward parabolic equation \begin{equation}\label{backwardcgoeq}\begin{cases}
	-\partial_t v - \Delta v +\phi\cdot\nabla v+ qv=0\quad &\text{ in }Q,\\
 v(x,T)=0 &\text{ in }\Omega,\end{cases}
\end{equation} 
of the form 
\begin{equation}\label{cgo2}
		v=e^{-\psi}\left(\chi(t)\exp\left(-\frac{1}{2}\int_0^\infty \zeta\cdot\phi(x+s\xi)\,ds\right) +z_-(x,t)\right)
	\end{equation}
	where $z_-(x,t;\phi,q)\in H^1(0,T;H^{-1}(\Omega))\cap L^2(0,T;H^1(\Omega))$ satisfies the following condition:
	\begin{equation}\label{decay2}
		\lim_{\rho\to\infty}\norm{z_-(x,t;\phi,q)}_{L^2(Q)}=0.
	\end{equation}
	\end{enumerate}
\end{theorem}

Since we mainly work in H\"{o}lder space and we only get exponentially growing solutions in $H^1(\Omega)$, we still need the following denseness property.

\begin{lemma}[Runge approximation with full data]\label{RungeDenseApprox}
   Suppose $\phi\in [C^1(Q)]^n$ and $q\in L^{\infty}(Q)$. Then for any solutions $w,v\in L^2(0,T;H^1(\Omega))\cap H^1(0,T;H^{-1}(\Omega))$	to \begin{equation}\label{forwardcgoeqapprox}\begin{cases}
	\partial_t w - \Delta w -\phi\cdot\nabla w+ qw=0\quad &\text{ in }Q,\\
 w(x,0)=0 &\text{ in }\Omega,\end{cases}
\end{equation} 
and 
\begin{equation}\label{backwardcgoeqapprox}\begin{cases}
	-\partial_t v - \Delta v +\phi\cdot\nabla v+ qv=0\quad &\text{ in }Q,\\
 v(x,T)=0 &\text{ in }\Omega,\end{cases}
\end{equation} respectively, 
and  any $\eta,\eta'>0$, there exist solutions $W,V\in C^{2+\alpha,1+\alpha/2}(\overline{Q})$ to \begin{equation}\label{forwardcgoeqapproxV}\begin{cases}
	\partial_t W - \Delta W -\phi\cdot\nabla W+ qW=0\quad &\text{ in }Q,\\
 W(x,0)=0 &\text{ in }\Omega,\end{cases}
\end{equation} 
and 
\begin{equation}\label{backwardcgoeqapproxW}\begin{cases}
	-\partial_t V - \Delta V +\phi\cdot\nabla V+ qV=0\quad &\text{ in }Q,\\
 V(x,T)=0 &\text{ in }\Omega,\end{cases}
\end{equation}  respectively
	such that
	\[\norm{W-w}_{L^2(Q)}<\eta\quad \text{ and }\quad \norm{V-v}_{L^2(Q)}<\eta.\]
\end{lemma}
\begin{proof}
    The proof mainly follows that of Lemma 4.1 of \cite{LinLiuLiuZhang2021-InversePbSemilinearParabolic-CGOSolnsSuccessiveLinearisation}. We will show the case for the forward parabolic equation, and the backward one can be proved similarly. Define
	\[X=\left\{W\in C^{2+\alpha,1+\alpha/2}(\overline{Q})\,\Big|\, 
	W \text{ is a solution to } \eqref{forwardcgoeq} \right\}\] and
	\[Y=\left\{w\in L^2(0,T;H^1(\Omega))\cap H^1(0,T;H^{-1}(\Omega)) \,\Big|\, w \text{ is a solution to } \eqref{forwardcgoeq} \right\}.\] 
	We aim to show that $X$ is dense in $Y$. By the Hahn-Banach theorem, it suffices to prove the following statement: If $f\in L^2(Q)$ satisfies
	$$\int_{Q}fW\,dxdt=0, \quad \text{ for any } W\in X,$$
	then
	$$\int_{Q}fw\,dxdt=0, \quad \text{ for any }w\in Y.$$
	To this end, 
	let $f\in L^2(Q)$ and suppose $\int_{Q}fW\,dxdt=0$,  for any  $W\in X$. Consider     
    \begin{equation}\begin{cases}
	-\partial_t \tilde{W} - \Delta \tilde{W} +\phi\cdot\nabla \tilde{W}+\tilde{W}\nabla\cdot\phi+ q\tilde{W}=f\quad &\text{ in }Q,\\
    \tilde{W}(x,t)=0\quad&\text{ on }\Sigma,\\
 \tilde{V}(x,T)=0 &\text{ in }\Omega,\end{cases}
\end{equation} 
with solution in $H^{2, 1}(Q)$. For any $W\in X$, one has 
\begin{align*}
	0=&\int_{Q}fW\,dx\,dt=\int_Q (-\partial_t \tilde{W} - \Delta \tilde{W} +\phi\cdot\nabla \tilde{W}+\tilde{W}\nabla\cdot\phi+ q\tilde{W}) W\,dx\,dt=\int_{\Sigma} \p_\nu \overline{W}W\, dS\,dt.
\end{align*}
Since $W|_{\Sigma}$ can be arbitrary function, which  is compactly supported on $\Sigma$, we must have $\partial_\nu \overline{W}=0$ on $\Sigma$.
Thus, for any $ w\in Y$,
\begin{align*}
	\int_{Q}fw \,dx\,dt&=\int_Q (-\partial_t \tilde{W} - \Delta \tilde{W} +\phi\cdot\nabla \tilde{W}+\tilde{W}\nabla\cdot\phi+ q\tilde{W}) w\,dxdt=\int_{\Sigma}\p_\nu\overline{W}w\,dSdt=  0,
\end{align*}
which verifies the assertion.
\end{proof}

\subsection{Recovery of Stationary Solution and Hamiltonian}

Let $(u_{0,i},m_{0,i})$ be the solutions of \eqref{eq:MFG2i} which satisfy \eqref{MFG2Linear0} with corresponding Hamiltonian $A_i(x)$ and running cost $F_i(x,m)$, for $i=1,2$, i.e. 
\begin{equation}\label{MFG2Linear0i}
    \begin{cases}
        -\Delta u_{0,i}(x) + [\nabla u_{0,i}(x)]^TA_i(x) \nabla u_{0,i}(x)= 0&\quad \text{in }\Omega,\\
        -\Delta m_{0,i}(x) - 2\nabla\cdot\left( m_{0,i}(x)[\nabla u_{0,i}(x)]^TA_i(x)\right)= 0  &\quad \text{in }\Omega.
    \end{cases}
\end{equation}
At the same time, the corresponding first order linearized solutions $(u^{(1)}_i,m^{(1)}_i)$ satisfy \eqref{MFG2Linear1} with the same Hamiltonian $A_i(x)$ and running cost $F_i(x,m)$, for $i=1,2$, i.e. 
\begin{equation}\label{MFG2Linear1i}
    \begin{cases}
        -\partial_t u^{(1)}_i(x,t) -\Delta u^{(1)}_i(x,t) + [\nabla u_{0,i}(x)]^TA_i(x) \nabla u^{(1)}_i(x,t)
        \\\quad\quad\quad\quad\quad\quad\quad\quad\quad\quad\quad\quad + [\nabla u^{(1)}_i(x,t)]^TA_i(x) \nabla u_{0,i}(x)= 0 &\quad \text{in }Q,\\
        \partial_t m^{(1)}_i(x,t) -\Delta m^{(1)}_i(x,t) - 2\nabla\cdot\left( m_{0,i}(x)[\nabla u^{(1)}_i(x,t)]^TA_i(x)\right) \\\quad\quad\quad\quad\quad\quad\quad\quad\quad\quad\quad\quad -2\nabla \cdot \left(m^{(1)}_i(x,t)[\nabla u_{0,i}(x)]^TA_i(x)\right) = 0  &\quad \text{in }Q,\\
        u^{(1)}_i(x,t)=g_1,\quad m^{(1)}_i(x,t)=h_1 &\quad \text{in }\Gamma,\\
        u^{(1)}_i(x,T)=0, \quad m^{(1)}(x,0) = 0  &\quad \text{in }\Omega.
    \end{cases}
\end{equation}

Observe that we have the identity
\[[\nabla u_{0,i}]^TA_i \nabla u^{(1)}_i+[\nabla u^{(1)}_i]^TA_i \nabla u_{0,i}=\left(\left[[\nabla u_{0,i}]^TA_i\right]^T+A_i\nabla u_{0,i}\right)\cdot\nabla u^{(1)}_i = 2A_i\nabla u_{0,i}\cdot \nabla u^{(1)}_i. \]
Let $\tilde{u}=u^{(1)}_1-u^{(1)}_2$, $\tilde{m}=m^{(1)}_1-m^{(1)}_2$ and $q_i=2A_i\nabla u_{0,i}$. Then, taking the difference of the two equations for $i=1,2$ for the first equation in \eqref{MFG2Linear1i}, we have 
\begin{equation}\label{MFG2Linear1U}
    \begin{cases}
        -\partial_t \tilde{u} -\Delta \tilde{u} + q_1\cdot \nabla \tilde{u} + (q_1-q_2)\cdot\nabla u^{(1)}_2= 0 &\quad \text{in }Q,\\
        \tilde{u}(x,t)=0,&\quad \text{in }\Gamma,\\
        \tilde{u}(x,T)=0&\quad \text{in }\Omega.
    \end{cases}
\end{equation}
Furthermore, by Theorem \ref{thm:cgo}, we have a solution of $u_2^{(1)}$ to \eqref{MFG2Linear1i} of the form \eqref{cgo2}.

Consider the solution $w$ of the form \eqref{cgo1} to 
\begin{equation}\label{TestFunction}
    \begin{cases}
        \partial_t w - \Delta w - q_1\cdot \nabla w+w\nabla\cdot q_1= 0 &\quad \text{in }Q,\\
        w(x,0)=0&\quad \text{in }\Omega.
    \end{cases}
\end{equation}
Multiplying \eqref{MFG2Linear1U} by $w$ and using integration by parts, we have
\[ \int_Q\tilde{u}\partial_t w -\int_Q \tilde{u}\Delta w - \int_Q \tilde{u}(q_1\cdot\nabla w + w\nabla\cdot q_1) + \int_Q w(q_1-q_2)\cdot\nabla u^{(1)}_2(x,t)= 0.\] 
When $\mathcal{M}_{F_1,A_1,U_1}=\mathcal{M}_{F_2,A_2,U_2}$, this gives 
\begin{equation}\label{IntegralIdentity}\int_Q w(q_1-q_2)\cdot\nabla u^{(1)}_2(x,t)\leq C\rho^{\frac{1}{2}},\end{equation}  by following the argument of Section 5 of \cite{ParabolicConvectionVectorCGO}. 
This means that when $\rho\to\infty$, 
\begin{equation}\label{IntegralIdentity2}\frac{1}{\rho}\int_Q w(q_1-q_2)\cdot\nabla u^{(1)}_2(x,t)\to0.\end{equation}

Substituting into \eqref{IntegralIdentity2} the expressions for $u_2^{(1)}$ and $w$ from \eqref{cgo2} and \eqref{cgo1} respectively, in the limit $\rho\to\infty$, we have 
\[
\int_Q \chi^2(t)e^{-i(x,t)\cdot(\xi,\tau)} (q_1-q_2)\cdot\zeta\,dx\,dt= 0\quad\text{ for all }\zeta\in\mathbb{S}^{n-1}
\]
Since the above identity holds for all $\chi\in C_{c}^{\infty}(0,T)$, therefore we get 
\[
\int_{\mathbb{R}^n} e^{-ix\cdot\xi} (q_1-q_2)\cdot\zeta\,dx= 0\quad\text{ for all }\zeta\in\mathbb{S}^{n-1}.
\]
Identifying this as the Fourier transform with respect to $\xi$, and since this holds for all unit vectors $\zeta$, this gives that 
\begin{equation}\label{qEq}q_1=q_2\quad \text{ in }\Omega,\end{equation} and we denote $q=q_1=q_2$.

Next, we consider \eqref{MFG2Linear0i}. Then the first equation can be rewritten as 
\begin{equation}
        -\Delta u_{0,i}(x) + \frac{1}{2}q(x)\cdot \nabla u_{0,i}(x)= 0\quad \text{in }\Omega.
\end{equation}
Observe that this is a second order elliptic equation in $u_{0,i}$. Since $\mathcal{M}_{F_1,A_1,U_1}=\mathcal{M}_{F_2,A_2,U_2}$, in particular $u_{0,1}=u_{0,2}$ and $\partial_\nu u_{0,1}=\partial_\nu u_{0,2}$ on $\partial\Omega$, by the unique continuation principle for elliptic operators \cite{KochTataru2001UCPElliptic}, it must hold that 
\begin{equation}\label{uEq}u_{0,1}=u_{0,2} \quad \text{ in }\Omega,\end{equation} and we denote $u_0=u_{0,1}=u_{0,2}$.

In particular, it holds that $\nabla u_{0,1}=\nabla u_{0,2}$. By Definition \ref{def:A}, for $A_1=\kappa_1g$ and $A_2=\kappa_2g$, \eqref{qEq} can be rewritten as \[\kappa_1g\nabla u_{0,1}=\kappa_2g\nabla u_{0,2},\] from which we obtain \[\kappa_1=\kappa_2.\] This means that up to the conformal class $C_g$, 
\[A(x):=A_1(x)=A_2(x)\quad\text{ in }Q.\]

Therefore, the second equation in \eqref{MFG2Linear0i} can now be rewritten as 
\begin{equation}\label{MFG2Linear0M}
    -\Delta m_{0,i}(x) - \nabla\cdot\left(q(x) m_{0,i}(x)\right)= 0  \quad \text{in }\Omega.
\end{equation}
Once again, this is a second order elliptic equation, this time in $m_{0,i}$. By the equality of the measurement maps $\mathcal{M}_{F_1,A_1,U_1}=\mathcal{M}_{F_2,A_2,U_2}$, in particular $m_{0,1}=m_{0,2}$ and $\partial_\nu m_{0,1}=\partial_\nu m_{0,2}$ on $\partial\Omega$, once again applying the unique continuation principle for elliptic operators, it must hold that 
\begin{equation}\label{mEq}m_0:=m_{0,1}=m_{0,2} \quad \text{ in }\Omega.\end{equation}

Substituting these results into \eqref{MFG2Linear1i}, we have that $\left(u_i^{(1)},m_i^{(1)}\right)$ satisfy a heat-type equation for $i=1,2$. When $\mathcal{M}_{F_1,A_1,U_1}=\mathcal{M}_{F_2,A_2,U_2}$, by the uniqueness of solutions for heat equations, it must hold that 
\begin{equation}\label{UM1Eq}u_1^{(1)}=u_2^{(1)}\quad \text{ and }m_1^{(1)}=m_2^{(1)}\quad \text{ in }Q.\end{equation}

\subsection{Recovery of Running Cost}

Now that we have obtained the equality of the stable stationary solutions $U=(u_0,m_0)$ and the uniqueness of Hamiltonian $A$, we can proceed with the unique identifiability results for $F$. The main ingredient for the recovery of $F$ is the following unique continuation principle, which is Theorem 6.1 of \cite{LiuLoZhang2024decodingMFG}:
\begin{theorem}\label{UCP}

Let $\Sigma'\subset \Sigma$ be an arbitrarily chosen
non-empty relatively open sub-boundary. For $\beta\in L^\infty(Q)$, assume that $(u_i,m_i) \in H^{2,1}(Q)\times H^{2,1}(Q)$ 
satisfy, 
\begin{equation}\label{eq:UCPeq}
    \begin{cases}
        -\partial_t u_i(x,t) -\mathcal{L}_1 u_i(x,t) = D_i(x,t)+\beta(x,t)m_i(x,t) &\quad \text{in }Q,\\
        \partial_t m_i(x,t) -\mathcal{L}_2 m_i(x,t) + \mathcal{L}_3 u_i(x,t) = E_i(x,t)  &\quad \text{in }Q,
    \end{cases}
\end{equation}
for $i=1,2$, for some regular second-order elliptic operators $\mathcal{L}_1$, $\mathcal{L}_2$ and $\mathcal{L}_3$, such that the coefficients (which depend on $x$ and $t$) of $\mathcal{L}_3$ are bounded,
and
\[\begin{cases}
u_i, \nabla u_i, \Delta u_i \in L^{\infty}(Q),\\
m_i, \nabla m_i \in L^{\infty}(Q), \quad 
\partial_t(u_1-u_2),\partial_t(m_1-m_2) \in L^2((\Sigma\backslash\Sigma')\times(0,T)).
\end{cases}\]
Then $u_1=u_2$, $\nabla u_1 = \nabla u_2$, $m_1=m_2$ and $\nabla m_1
= \nabla m_2$ on $\Sigma' \times (0,T)$ implies $u_1=u_2$ and $m_1=m_2$ 
in $Q$.

\end{theorem}

With this result in hand, we can proceed to recover the running cost $F$, by making use of the higher orders of linearization. 

    Consider the second order linearization.
    Let $\bar{u}=u^{(1,2)}_1-u^{(1,2)}_2$ and $\bar{m}=m^{(1,2)}_1-m^{(1,2)}_2$. Then, since we have obtained the equality of the first order linearized solutions in \eqref{UM1Eq}, by taking the difference of the two second order linearized systems \eqref{MFGQuadraticLinear2} for $i=1,2$, we have that $(\bar{u},\bar{m})\in [H^{2,1}(Q)]^2$ is the solution to the system 
    \begin{equation}\label{eq:MFG2Diff2}
    \begin{cases}
        -\partial_t \bar{u}(x,t) -\Delta \bar{u}(x,t) + 2q(x)\cdot\nabla \bar{u}(x,t) =  [F^{(2)}_1(x) - F^{(2)}_2(x)]m^{(1)}(x,t)m^{(2)}(x,t) & \quad \text{in }Q,\\
        \partial_t \bar{m}(x,t) -\Delta\bar{m}(x,t) -2\nabla\cdot\left( m_0(x)[\nabla \bar{u}(x,t)]^TA(x)\right) - \nabla\cdot\left( \bar{m}(x,t) q(x)\right) = 0  &\quad \text{in }Q,\\
        \bar{u}(x,T)=0&\quad \text{in }\Omega,\\
        \bar{m}(x,0)=0 &\quad \text{in }\Omega,\\
        \bar{u}=\nabla \bar{u} = \bar{m} = \nabla \bar{m} = 0 &\quad \text{on }\Sigma,
    \end{cases}
    \end{equation}
    when $\mathcal{M}_{F_1,A_1,U_1}=\mathcal{M}_{F_2,A_2,U_2}$.

    Since $m^{(1)},m^{(2)}$ have been uniquely obtained, expanding $\nabla\cdot\left( m_0[\nabla \bar{u}]^TA\right)$ and $\nabla\cdot\left( \bar{m} q\right)$, we can view $D_i:=F^{(2)}_i(x)m^{(1)}(x,t)m^{(2)}(x,t)$ in \eqref{eq:UCPeq}, and apply the unique continuation principle given by Theorem \ref{UCP} to obtain \begin{equation}
        \bar{v}=\bar{m}\equiv0. 
    \end{equation}
    Substituting this into the first equation of \eqref{eq:MFG2Diff2}, we have that $D_1=D_2$. Choosing $m^{(1)},m^{(2)}\not\equiv0$, we obtain the result
    \[F^{(2)}_1(x)=F^{(2)}_2(x).\] 

    For the higher order Taylor coefficients of $F$, we make use of the high order linearized parabolic systems. The most important ingredient is that the non-linear terms in higher order systems only depend on the solutions of lower order terms. Therefore, we can apply mathematical induction and repeat similar argument to show that $F^{(k)}_1(x)=F^{(k)}_2(x)$ for all $k\in\mathbb{N}$. Hence we have the unique identifiability for the source functions, i.e. \[F_1=F_2.\]
    
    The proof is complete.

\vspace{1em}

\noindent\textbf{Acknowledgment.} 
	The work was supported by the Hong Kong RGC General Research Funds (No. 11311122, 11304224 and 11300821), the NSFC/RGC Joint Research Fund (No. N\_CityU101/21), and the ANR/RGC Joint Research Grant (No. A\_CityU203/19).

    \vspace{1em}

\noindent\textbf{Declaration of interests.} I have nothing to declare.

\bibliographystyle{plain}
\bibliography{ref1,refInversePb,refMFG}
\end{document}